\documentclass[preprint,12pt]{elsarticle}

\usepackage{amssymb,verbatim}
 \usepackage{amsthm}
 
\newtheorem{lemma}[equation]{Lemma}
\newtheorem{theorem}[equation]{Theorem}
\newtheorem{definition}[equation]{Definition}
\newtheorem{conj}[equation]{Conjecture}
\newtheorem{cory}[equation]{Corollary}

\journal{a Journal}

\begin{document}

\begin{frontmatter}



\title{Some Exceptional Beauville Structures}


\author{Ben Fairbairn}

\address{Department of Economics, Mathematics and Statistics, Birkbeck, University of London, Malet Street, London, WC1E 7HX 

 bfairbairn$@$ems.bbk.ac.uk}

\begin{abstract}
We first show that every quasisimple sporadic group, including the Tits group $^2F_4(2)'$, possesses an unmixed strongly real Beauville structure aside from the Mathieu groups $M_{11}$ and $M_{23}$ (and possibly 2\.{}$\mathbb{B}$ and $\mathbb{M}$). We go on to show that no almost simple sporadic group possesses a mixed Beauville structure. We go on to use the exceptional nature of the alternating group $A_6$ to give a strongly real Beauville structure for this group explicitly correcting an earlier error of Fuertes and Gonz\'{a}lez-Diez. In doing so we complete the classification of alternating groups that possess strongly real Beauville structures. We conclude by discussing mixed Beauville structures of the groups $A_6:2^{(2)}$.
\end{abstract}


\end{frontmatter}


\section{Introduction}

\subsection{Beauville Surfaces, Structures and Groups}

Complex surfaces lie in the intersection of algebraic geometry, differential geometry and complex variable theory and as such enjoy applications as far afield as number theory, topology and even superstring theory. Finding examples of such surfaces that are easy to work with is thus more important than ever. One approach that has proved particularly fruitful over the past ten years or so is the concept of a \emph{Beauville surface}: a class of 2-dimensional complex algebraic varieties that
are rigid, in the sense of admitting no non-trivial deformations, whose study was recently initiated by Bauer, Catanese and Grunewald in \cite{2, 3, 5}. These surfaces are defined over the field $\overline{\mathbb{Q}}$ of
algebraic numbers, providing a geometric action of the absolute Galois group $Gal(\overline{\mathbb{Q}}/\mathbb{Q})$.
By generalizing Beauville's original example \cite[p.159]{4}, they can be constructed from finite
groups acting on suitable pairs of algebraic curves, and here we give some new examples
of surfaces of this kind.

\begin{definition}
A \emph{Beauville surface of unmixed type} is a compact complex surface $\mathcal{S}$ such that

(a) $\mathcal{S}$ is isogenous to a higher product, that is, $\mathcal{S}\cong(\mathcal{C}_1\times\mathcal{C}_2)/G$ where $\mathcal{C}_1$ and $\mathcal{C}_2$ are algebraic curves of genus at least 2 and $G$ is a finite group acting by the diagonal action freely on $\mathcal{C}_1\times\mathcal{C}_2$ by holomorphic transformations; 

(b) If $G_0<G$ denotes the subgroup consisting of the elements which preserve each of the
factors, then $G_0$ acts effectively on each curve $\mathcal{C}_i$ so that $\mathcal{C}_i/G_0\cong
\mathbb{P}^1$ and $\mathcal{C}_i\rightarrow\mathcal{C}_i/G_0$ ramifies over three points.
\end{definition} 

Condition (b) is equivalent to
each curve $\mathcal{C}_i$ admitting a regular dessin in the sense of the theory of dessin
d'enfants due to Grothendieck  \cite{6, 11, 23}, or equivalently an orientably regular hypermap \cite{14}, with $G$ acting as
the orientation-preserving automorphism group.

One particularly attractive feature of this class of curves is the fact that the above definition can be translated into more finitery combinatorial terms that `internalize' the structure of the surface into the group $G$ in the following way.

\begin{definition}\label{maindef}
Let $G$ be a group. An \emph{unmixed Beauville structure} of $G$ is a pair of generating sets $\{x_i,y_i,z_i\}\subset G$ with $l_i:=o(x_i)$, $m_i:=o(y_i)$ and $n_i:=o(z_i)$ for $i=1,2$ such that the following holdsfor $i=1,2$.
\begin{enumerate}
\item $x_iy_iz_i = 1$;
\item $l_i^{-1} + m^{-1}_i + n^{-1}_i < 1$ and
\item Defining
$$\Sigma(x_i,y_i,z_i):=\bigcup_{g\in G}\bigcup_{j=0}^{\infty}\{gx_i^jg^{-1},gy_i^jg^{-1},gz_i^jg^{-1}\}$$
we have
$$\Sigma(x_1,y_1,z_1)\cap\Sigma(x_2,y_2,z_2)=\{e\}.$$
We say that the Beauville structure has \emph{type} $((l_1,m_1,n_1),(l_2,m_2,n_2))$.  We call a group possessing an unmixed Beauville structure a \emph{Beauville group}.
\end{enumerate}
\end{definition}

It was conjectured by Bauer, Catanese and Grunewald that every nonabelian finite simple group is a Beauville group, with the sole exception of the alternating group A$_5$ \cite[Conjecture 1]{3}. Several authors settled special cases of this conjecture \cite{3,FG,FJ,GP}. Finally the full conjecture was recently verified by the author, Magaard and Parker in \cite[theorem 1.3]{FMP} where we prove the following more general theorem.

\begin{theorem}\label{FMPref}
Aside from the groups SL$_2$(5) and PSL$_2(5)(\cong$A$_5)$ every nonabelian finite quasisimple group possesses an unmixed Beauville structure.
\end{theorem}

Similar results were obtained at around the same time by Guralnick and Malle in \cite{GM} and by Garion, Larsen and Lubotzky in \cite{GLL}.

\subsection{Real Surfaces and Mixed Structures}

Now that we know that almost every quasisimple group is a Beauville group, we are in a position to address the more general issue of \emph{what these Beauville structures and surfaces actually look like}. One more specific instance of this somewhat vague question is to ask when a complex surface $\mathcal{S}$ is `real' (ie there is a biholomorphic map $\sigma:\mathcal{S}\rightarrow\overline{\mathcal{S}}$ such that $\sigma^2$ is the identity). As is the `zeitgeist' of Beauville constructions, this topological property can be translated into finitery combinatorial terms inside the corresponding group.

\begin{definition}
We say that a Beauville surface is \emph{real} and its corresponding Beauville structure $\{x_i,y_i,z_i|i=1,2\}$ and group $G$ are \emph{strongly real} if there exist automorphisms $\phi_i\in Aut(G)$ for $i=1,2$ that differ only in an inner automorphism of $G$ such that $x_i^{\phi_i}=x_i^{-1}$ and $y_i^{\phi_i}=y_i^{-1}$ for $i=1,2$. If such an automorphism exists we say that $G$ is \emph{strongly} $(l_i,m_i,n_i)$ \emph{generated}.
\end{definition}

It has been conjectured by Bauer, Catanese and Grunewald \cite[Conjecture 3]{2} that all but finitely many of the finite simple groups are strongly real Beauville groups. Given the progress made on the wider class of quasisimple groups in theorem \ref{FMPref}, it seems natural to make following stronger conjecture.

\begin{conj}
All but finitely many finite quasisimple groups possess strongly real unmixed Beauville structures.
\end{conj}

Clearly settling the status as strongly real Beauville groups of the sporadic groups makes no impact whatsoever on this conjecture, since there are only finitely many sporadic groups. (Note that throughout we shall use the term \emph{sporadic group} to refer to the 26 `traditional' sporadic groups as well as the Tits group $^2F_4(2)'$.) Nonetheless it is the opinion of the author that settling these questions for the sporadic groups is no less important than, for example, determining their symmetric genii or determining which of them are Hurwitz groups \cite{ConderWW,MonsterHur}. It is, however, arguably \emph{more} useful to settle this matter for the sporadic groups since curves and surfaces associated with them are likely to be very exceptional in nature and much of the original motivation for the study of Beauville surfaces was for their use as counterexamples \cite{2}.

\begin{theorem}\label{maintheorem}
(a) The Mathieu groups M$_{11}$ and M$_{23}$ are not strongly real Beauville groups.

(b) Every other quasisimple sporadic group (except possibly the groups 2\.{}$\mathbb{B}$ and $\mathbb{M}$) is a strongly real Beauville group.
\end{theorem}

Our computationally intensive methods are unable to handle the groups 2\.{}$\mathbb{B}$ and $\mathbb{M}$ (though we are able to show that the simple group $\mathbb{B}$ is a strongly real Beauville group). We make no apologies for this: to resolve the similar problem of settling the Monster's status as a Hurwitz group took almost ten years of computing time \cite[p.370]{MonsterHur}! Nonetheless, the vast majority of the conjugacy classes in both of these groups are strongly real (see \cite{Sul}) and since the only problem the groups M$_{11}$ and M$_{23}$ encounter is a lack of strongly real classes (see Section \ref{non}) we make the following conjecture. 

\begin{conj}
Both of the groups 2\.{}$\mathbb{B}$ and $\mathbb{M}$ are strongly real Beauville groups. 
\end{conj}

In the case of $\mathbb{M}$ we make several specific conjectural remarks concerning how a strongly real Beauville structure for $\mathbb{M}$ might be obtained in Section 5 - the problem is not a lack of theoretical ideas or knowledge about the monster, but simply a lack of computational power!

A mixed Beauville structure is a Beauville structure in which the action of our group interchanges the two curves defining our surface. This can also be `internalized' into $G$, though we postpone this definition until Section 6 where we prove the following theorem.

\begin{theorem}\label{unmixed}
Let $G$ be an almost simple group such that the derived subgroup $[G,G]$ is sporadic. Then $G$ does not possess an mixed Beauville structure.
\end{theorem}

Finally, our attention turns to the question of which alternating groups possess a strongly real Beauville structure and in doing so we prove the following theorem.

\begin{theorem}\label{alt}
The alternating group $A_6$ has a strongly real Beauville structure of type ((4,4,4),(5,5,5)).
\end{theorem}
When combined with the structures explicitly constructed in the proof of \cite[theorem 2]{FG} we have the following corollary.
\begin{cory}\label{cor}
The alternating group $A_n$ is a strongly real unmixed Beauville group if
and only if $n\geq6$.
\end{cory}
(Note that in \cite[theorem 2]{FG} Fuertes and Gonz\'{a}lez-Diez claim to prove the above result with the bound $n\geq6$ replaced with $n\geq7$ - an error that the above theorem explicitly corrects. Interestingly, this correction requires the use of the exceptional nature of $Aut(A_6)$, so is clearly very different to the $n\geq7$ cases. We further note that in \cite[theorem 3.1]{FJ} strongly real Beauville structures for the groups $PSL_2(q)$ are obtained. This appears to also correct the above error since $PSL_2(9)\cong A_6$, but it is only by delegating this case to the reader that they achieve this. Given that this strongly real Beauville structure can only be constructed by using an automorphism that is exceptional, regardless of whether this group is viewed as $PSL_2(9)$ or A$_6$, explicitly resolving this case is clearly desirable.) We conclude with a brief discussion of mixed Beauville structures of groups of the form $A_6:2^{(2)}$.





\section{Strongly Real Sporadic Beauville groups}\label{Main}

\subsection{Our Construction}\label{construction}

Roughly speaking, our method of showing that a group is a strongly real Beauville group, which in principal may be applied to any perfect group of even order that possesses a strongly real Beauville structure, is as follows. 

We recall the following from \cite{Bray}. Let $t,g\in G$ be such that $t$ is an involution.
\begin{itemize}
\item If $o(tt^g)=2r$ for some integer $r$ then $(tt^g)^r\in C_G(t)$.
\item If $o(tt^g)=2r+1$ for some integer $r$ then $g(tt^g)^r\in C_G(t)$.
\end{itemize}
Let $x_i:=tt^{g_i}$ for some elements $g_i\in G$ for $i=1,2$. The above observations makes it easy to find some element $u\in G$ that commutes with $t$ and does not normalize the subgroup $\langle x_1\rangle$. We can then define the elements $y_i:=(x_i^{j(i)})^u$ for $i=1,2$, the value of the integer $j(i)$ being chosen to make the product $x_iy_i$ `nice' (ie we ensure that the conditions of definition \ref{maindef} are satisfied and when necessary hopefully makes it easier to see from the subgroup structure of $G$ that $\langle x_i,y_i\rangle=G$). This gives a Beauville structure for $G$. Since $u$ and $t$ commute we also have that $x_i^t=x_i^{-1}$ and $y_i^t=y_i^{-1}$ thus the Beauville structure just constructed must be strongly real.


\subsection{The Structures}

In this section we describe and tabulate the Beauville structures that we construct here.

Standard generators for the sporadic groups are given on the online atlas \cite{onlineATLAS} and are named $a$ and $b$. In each case we have that $o(a)=2$, so where possible we use this element to define the automorphisms needed when constructing strongly real Beauville structures. For background information on standard generators more generally see the original article by Wilson \cite{WilsonStandardGenerators}.

The types of the Beauville structures we construct here are given in Table \ref{sporadic1}. The words used to define our Beauville structures are given in Table \ref{sporadic2}. We remark that whilst it is common to use lower case letters for the standard generators of a simple group and upper case letters for their covering groups. For the sake of aesthetics we use lower case letters in both cases, it being clear which are the non-simple cases.

In some cases it is either necessary or desirable to use an involution other than $a$ that we call $c$. The words in the standard generators used to define these elements $c$ are given in Table \ref{sporadic3}. In each case the fact that the given elements generate may be verified using either permutation or matrix representations of these groups available on \cite{onlineATLAS}, either directly or by the observations made in the next section.

\begin{table}
\begin{center}
\begin{tabular}{|c|c|c|c|}
\hline
$G$&Type&$G$&Type\\
\hline
J$_1$&((19,19,11),(15,15,7))&3\.{}O'N&((28,28,12),(19,19,19))\\
2\.{}M$_{12}$&((5,5,3),(11,11,11))&Co$_3$&((7,7,23),(5,5,24))\\
12\.{}M$_{22}$&((5,5,5),(12,12,6))&Co$_2$&((16,16,8),(11,11,7))\\
2\.{}J$_2$&((7,7,7),(12,12,8))&6\.{}Fi$_{22}$&((7,7,5),(13,13,13))\\
$^2$F$_4(2)'$&((5,5,5),(4,4,4))&HN&((5,5,5),(6,6,6))\\
2\.{}HS&((15,15,5),(8,8,7))&Ly&((67,67,40),(37,37,21))\\
3\.{}J$_3$&((17,17,19),(9,9,9))&Th&((19,19,19),(13,13,13))\\
M$_{24}$&((5,5,5),(6,6,11))&Fi$_{23}$&((5,5,5),(6,6,4))\\
3\.{}McL&((5,5,5),(6,6,6))&2\.{}Co$_1$&((5,5,5),(6,6,6))\\
He&((3,3,6),(17,17,17))&J$_4$&((43,43,11),(29,29,6))\\
2\.{}Ru&((4,4,10),(13,13,7))&3\.{}Fi$'_{24}$&((9,9,9),(11,11,26))\\
6\.{}Suz&((13,13,13),(12,12,10))&$\mathbb{B}$&((13,13,19),(12,12,20))\\
\hline
\end{tabular}
\end{center}
\medskip

\caption{The types of the Beauville structures defined by the words given in Tables \ref{sporadic2} and \ref{sporadic3}. See definition \ref{maindef}.} 
\label{sporadic1}
\end{table}

\begin{table}
\begin{center}
\begin{tabular}{|c|c|c|c|c|c|}
\hline
$G$&$x_1$&$x_2$&$u$&$j(1)$&$j(2)$\\
\hline
J$_1$&$aa^b$&$aa^{bab}$&$b(ab^2ab)^9$&8&10\\
\textbf{2\.{}M}$_{12}$&$aa^{(ba)^2b^2ab^2}$&$aa^{(bab)^2ab}$&$bab(a(bab)^{23}(ab)^2)^2$&4&3\\
12\.{}M$_{22}$&$aa^{bab^2}$&$aa^b$&$b^2ab(ab(b^2a)^3b)^2$&3&5\\
\textbf{2\.{}J}$_2$&$aa^{(b^2a)^2b^4}$&$aa^{(ba)^2b^4}$&$$$(a(ab^2)^{23}a^2b^2)^6$&1&3\\
$^2$F$_4(2)'$&$aa^{bab^2}$&$aa^{babab^2}$&$b(ab^2ab)^2$&1&1\\
\textbf{2\.{}HS}&$cc^{b^3a}$&$cc^{ab^3ab}$&$(cb^3cb^2)^6$&14&1\\
\textbf{3\.{}J}$_3$&$aa^{(ba)^3b}$&$aa^b$&$b(ab^2ab)^4$&3&1\\
M$_{24}$&$aa^{bab}$&$aa^{bab^2}$&$(ab^2ab)^6$&1&1\\
3\.{}McL&$aa^{(ba)^2b}$&$aa^{(ba)^2b^4}$&$b(ab^4ab)^2$&1&1\\
He&$cc^{b^3}$&$cc^{bab}$&$(cb^6cb)^2$&1&13\\
2\.{}Ru&$cc^{aba}$&$cc^{abab^3a}$&$a(ca^3ca)^3$&1&2\\
\textbf{6\.{}Suz}&$cc^{(ba)^7}$&$cc^{b}$&$(ca)^{20}$&6&1\\
3\.{}O'N&$aa^{bab}$&$aa^{b^2ab^2}$&$(ab^3ab)^6$&2&12\\
Co$_3$&$aa^{(bab)^3b^2}$&$aa^{b^2abab^2ab}$&$b(ab^3ab)^7$&5&1\\
Co$_2$&$cc^{b^2}$&$cc^{b^4ab^2ab}$&$(cb^4cb)^{15}$&1&3\\
6\.{}Fi$_{22}$&$cc^{bab^7}$&$cc^{b^2ab^4}$&$(cb^{12}cb)^{10}$&6&6\\
HN&$aa^b$&$aa^{babab}$&$b(ab^2ab)^2$&1&1\\
Ly&$aa^{(ba)^2b}$&$aa^{(ba)^2b^2ab}$&$b(ab^4ab)^4$&3&2\\
Th&$aa^{(ba)^2b^2}$&$aa^{(ba)^4b^2ab}$&$(ab^2ab)^5$&1&3\\
Fi$_{23}$&$aa^{b^2ab}$&$aa^{(ba)^3b^2}$&$b(ab^2ab)^2$&2&1\\
2\.{}Co$_1$&$cc^{(ab^2ab)^2a}$&$cc^{ab^2(ab)^2a}$&$b(cb^2cb)^2$&2&1\\
J$_4$&$cc^{ab^2ab^3ab}$&$cc^{(ab^3)^4}$&$b(cb^3cb)^{15}$&1&1\\
3\.{}Fi$'_{24}$&$cc^b$&$cc^{(ab)^3}$&$b(cb^2cb)^4$&1&5\\
$\mathbb{B}$&$aa^{(ba)^2b^2}$&$aa^{(ba)^2b^2ab}$&$(ab^2ab)^{20}$&2&1\\
\hline
\end{tabular}
\end{center}\medskip

\caption{Words in terms of the standard generators defining a strongly real Beauville structure for the full covering group of each of the sporadic simple groups considered here. In each case the elements $a$ and $b$ are the standard generators. In cases where the use of an element labeled $c$ is required, words in the standard generators defining these elements are given in Table \ref{sporadic3}. In some cases it was necessary/desirable to use the standard generators for $G:2$ rather than $G$. These cases we write in bold font.} 
\label{sporadic2}
\end{table}

\begin{table}
\begin{center}
\begin{tabular}{|c|c|c|c|}
\hline
$G$&$c$&$G$&$c$\\
\hline
2\.{}HS&$(bab^2ab^4a)^5$&6\.{}Fi$_{22}$&$(abab^{10})^6$\\
He&$(ab^3)^4$&2\.{}Co$_1$&$(ab)^{20}$\\
2\.{}Ru&$b^2$&J$_4$&$(abab^2)^5$\\
6\.{}Suz&$(bab^2(ab)^2)^{28}$&3\.{}Fi$_{24}'$&$((ab)^4b)^{18}$\\
Co$_2$&$(ab(ba)^2b((bab)^2(ab^2)^2)^2)^3$&&\\
\hline
\end{tabular}
\end{center}\medskip

\caption{Words in terms of the standard generators $a$ and $b$ defining an involution $c$ in cases the involution $a$ is unable to define a strongly real Beauville structure via the construction described in Section \ref{construction}.} 
\label{sporadic3}
\end{table}
\subsection{Homomorphic images}

So far we have only shown that the full covering group of each of the groups of part ($b$) of theorem \ref{maintheorem} are strongly real Beauville groups. In this subsection we consider the cases of the quasisimple sporadic groups with non-trivial centers and their homomorphic images.

Given a group $G$, it is tempting to look for a Beauville structure in the quotient $G/N$ by some normal
subgroup $N\vartriangleleft G$, and to try to lift this back to $G$. However, a triple that generates
$G/N$ need not lift back to a triple generating $G$, and even if it does, the condition (2) of definition \ref{maindef} may
not be satisfied. In this situation, the following two lemmata  are of great use (whilst the proofs of these results may seem trivial to the group theorist we include references to their proofs for the sake of the less group theoretically inclined reader).

\begin{lemma}
If $G$ is a perfect group, $N$ is a central subgroup of G, and $S$ is a subset of G
such that the image of $S$ in $G/N$ generates $G/N$, then $S$ generates $G$.
\end{lemma}

\begin{proof}
See \cite[lemma 4.1]{FJ}.
\end{proof}

\begin{lemma}
Let $G$ have generating triples $(x_i, y_i, z_i)$ with $x_iy_iz_i = 1$ for $i = 1, 2$, and a
normal subgroup $N$ such that at least one of these triples is faithfully represented in $G/N$.
If the images of these triples correspond to a Beauville structure for $G/N$, then these triples
correspond to a Beauville structure for $G$.
\end{lemma}

\begin{proof}
See \cite[lemma 4.2]{FJ}.
\end{proof}

From the types of the Beauville structures obtained in the previous section and from the orders of the centers of the relevant groups it is clear that the above lemmata may be applied to the Beauville structures obtained in the previous section.

\section{Large Strongly Real Beauville Groups}

In this section we prove that the Beauville structures defined in Section \ref{Main} do indeed generate the groups claimed in the cases where the representations of the groups in question are too cumbersome for this to be verified directly. In doing so we complete the proof of part ($b$) of theorem \ref{maintheorem}. In each case it is taken for granted that the elements refered to from the previous section do indeed have the stated orders and we focus only on the question of generation in each case. Any direct calculation refered to in the below proofs may easily be performed in {\sc Magma} \cite{Magma} or GAP \cite{GAP}.

\begin{lemma}
The Harada-Norton group HN possosses a strongly real Beauville structure of type ((5,5,5),(6,6,6)).
\end{lemma}

\begin{proof}
From the list of maximal subgroups of HN, as listed in the {\sc Atlas} \cite[p.166]{ATLAS}, we see that no proper subgroup contains elements of order 22 and order 25. Direct computation shows that $o(x_1y_1x_1^2y_1^3)=o(x_2y_2^3x_2^2y_2^4)=22$ and $o(x_1y_1x_1^3y_1^4)=o(x_2y_2x_2^4x_2^5)=25$, hence $\langle x_1,y_1\rangle=\langle x_2,y_2\rangle=G$.
\end{proof}

\begin{lemma}
The Lyons group Ly possesses a strongly real Beauville structure of type ((67,67,40),(37,37,21)).
\end{lemma}

\begin{proof}
From the list of maximal subgroups of Ly, as listed in the {\sc Atlas} \cite[p.174]{ATLAS}, we see that an element of order 67 is contained in only one maximal subgroup, a copy of the Frobenious group 67:22. This clearly contains no elements of order 40. Similarly we see that an element of order 37 is contained in only one maximal subgroup, a copy of the Frobenious group 37:18. Since this clearly contains no elements of order 21 we must have $\langle x_1,y_1\rangle=\langle x_2,y_2\rangle=G$.
\end{proof}

\begin{lemma}
The Thompson group Th possesses a strongly real Beauville structure of type ((19,19,19),(13,13,13)).
\end{lemma}

\begin{proof}
From the list of maximal subgroups of Th, as listed in \cite{Linton1,Linton2} (note that list given in the {\sc Atlas} \cite[p.177]{ATLAS} is incomplete), we see that the only maximal subgroups containing elements of order 31 are isomorphic to either 2$^{5}$\.{}L$_5$(2) or the Frobenious group 31:15. These subgroups clearly contain no elements of order 19 or 13. Direct computation shows that $o(x_1y_1x_1^2y_1^4)=o(x_2y_2x_2^2y_2^{11})=31$, and so $\langle x_1,y_1\rangle=\langle x_2,y_2\rangle=G$.
\end{proof}

\begin{lemma} The Janko group J$_4$ possesses a strongly real Beauville structure of type ((43,43,11),(29,29,6)).
\end{lemma}

\begin{proof}
From the list of maximal subgroups of J$_4$, as listed in the {\sc Atlas} \cite[p.190]{ATLAS}, we see that an element of order 43 is contained in only one maximal subgroup, a copy of the Frobenious group 43:14. This clearly contains no elements of order 11. Similarly we see that an element of order 29 is contained in only one maximal subgroup, a copy of the Frobenious group 29:28. Since this clearly contains no elements of order 6 we must have $\langle x_1,y_1\rangle=\langle x_2,y_2\rangle=G$.
\end{proof}

\begin{lemma}
The Baby Monster $\mathbb{B}$ possesses a strongly real Beauville structure of type ((13,13,19),(12,12,20))
\end{lemma}

\begin{proof}
From the list of maximal subgroups of $\mathbb{B}$, as listed in \cite{Bmax} (note that list given in the {\sc Atlas} \cite[p.217]{ATLAS} is incomplete), we see that the only maximal subgroup containing elements of order 47 is isomorphic the Frobenious group 47:23. This subgroup clearly contains no elements of order 13 or 12. Direct computation shows that $o(x_1y_1^7x_1^3y_1^5)=o(x_2y_2x_2^3y_2^7)=47$, and so $\langle x_1,y_1\rangle=\langle x_2,y_2\rangle=G$.
\end{proof}

\section{Non-Strongly Real Beauville Groups}\label{non}

In this short section we prove that the sporadic groups M$_{11}$ and M$_{23}$ are not strongly real Beauville groups. In doing so we complete the proof of theorem \ref{maintheorem}. Note that the strongly real classes of the sporadic simple group were classified by Suleiman in \cite{Sul}.


\begin{lemma}\label{m23}
The groups M$_{11}$ and M$_{23}$ are not strongly real Beauville groups.
\end{lemma}

\begin{proof}
In both cases the only strongly real classes are classes of elements of order at most 6. In each group there is only one class of elements of order 2, one of order 3 and one of order 5. 


Computer calculations show that in each case, the group is only strongly $(5,5,m)$ generated if the integer $m$ is 4 or 6 and that neither group is strongly $(3,3,m)$ generated for any integer $m$.
\end{proof}

We remark that in \cite[p.35]{2} Bauer, Catanese and Grunewald state that they were unable to find a strongly real Beauville structure for M$_{11}$ (among other groups). The above lemma explains why.

\section{The Monster}
\begin{table}
\begin{center}
\begin{tabular}{|l|l|l|l|}
\hline
Element&Order&Element&Order\\
\hline
$p$&42&$s$&39\\
$qr$&19&$q^2r^2$&57\\
$qsr$&35&$qs^2r$&105\\
\hline
\end{tabular}
\end{center}\medskip
\caption{Some elements of $\mathbb{M}$ inverted by conjugation by $g$ and their orders.}
\label{monstertab1}

\end{table}
\begin{table}
\begin{center}
\begin{tabular}{|c|c|c|c|}
\hline
($x$,$y$,$xy$)&$(o(x),o(y),o(xy))$&($x$,$y$,$xy$)&$(o(x),o(y),o(xy))$\\
\hline
$(p,s,ps)$&(42,39,19)&$(qr,s,qrs)$&(19,39,22)\\
$(p^2,s,p^2s)$&(21,39,39)&$(qr,s^2,qrs^2)$&(19,39,66)\\
$(p^3,s,p^3s)$&(14,39,56)&$(p^2,qr,p^2qr)$&(21,19,60)\\
$(p,qr,pqr)$&(42,19,42)&$(p^2,s^2,p^2s^2)$&(21,39,55)\\
$(p,qsr,pqsr)$&(42,35,57)&$(qsr,s,qsrs)$&(35,39,105)\\
\hline
\end{tabular}
\end{center}\medskip
\caption{Some sets of elements of $\mathbb{M}$ that could potentially strongly $(a,b,c)$-generate the group.}
\label{monstertab2}
\end{table}

We give a brief discussion as to how a strongly real Beauville structure of the monster group $\mathbb{M}$ might be obtained.

In \cite{MonsterHur} Wilson proves that $\mathbb{M}$ can be generated by a pair of elements $g$ and $h$ such that $g$ is in class 2B, $h$ is in class 3B and $gh$ is in class 7B. In the process of proving this Wilson defines the following four elements
$$p=ghgh^2\mbox{, }q=ghghgh^2\mbox{, }r=ghgh^2gh^2\mbox{, }s=ghghgh^2gh^2.$$

Firstly, to apply our construction of Section 1.3 we need an involution of $\mathbb{M}$ - naturally we take the element $g$. 

The orders of several short words in the elements above are given in \cite[ Table 1]{MonsterHur}. In particular we have that $o(p)=42$. Now, for our Bray-type element, $u$, observe that $p^g=p^{-1}$ and so $g\not=p^{21}\in Z(\langle p,g\rangle)$. Whilst other short words in the above elements, such as those appearing in Table \ref{monstertab1}, are inverted by conjugation by $g$, these words often have odd order and so there is no guarantee that the involution produced will be distinct from $g$.

For our elements $x_i$, $i=1,2$ (which immediately give us the elements $y_i:=x_i^{u}$) we note that several of the words given in \cite[ Table 1]{MonsterHur} are inverted by conjugation by $g$, such as those given in Table \ref{monstertab1} and their powers, and any one of these provide candidates for our $x_i$s.

A slightly different approach is as following. In several cases, the products of elements found in Table \ref{monstertab1} also have their orders listed in \cite[ Table 1]{MonsterHur}. We can thus define at least one of our (potential) generating pairs by taking these elements themselves. We list a few of these possibilities in Table \ref{monstertab2}. 

We remark that proving that a proposed generating set $\mathbb{M}$ does in fact generate is easier than it first appears. Whilst the maximal subgroups of $\mathbb{M}$ have yet to be classified, a substantial amount of information is known. In particular, a complete classification of the maximal subgroups that contain elements of class 2A is known - see \cite{NW}. An immediate corollary of this classification is that the only maximal subgroups of $\mathbb{M}$ containing elements of order 94 are copies of 2\.{}$\mathbb{B}$. Finding a word in our set of proposed generators of order 94 forces the set to be contained in some copies of 2\.{}$\mathbb{B}$ and another word in our set of proposed generators that cannot lie in such a subgroup proves that the set generates. This is precisely how Wilson showed in \cite{MonsterHur} that the above $g$ and $h$ generate $\mathbb{M}$ - it turns out that $o(ppqsrpsrqsq)=94$ and $o(ppqsrqqrprq)=41$. 

Whilst multiplying elements of $\mathbb{M}$ together is extremely difficult, computing the order of such a word is somewhat easier - the method described in \cite{LPWW}, computing orders by analyzing orbits of specially chosen vectors in the natural 196882 dimensional $\mathbb{F}_2$ module, being the method used to calculate the orders given above.

\section{Mixed Beauville structures}

In this short section we consider the mixed case and prove theorem \ref{unmixed}. Recall that a Beauville structure is mixed if the action of the group interchanges the two curves being used to define the surface. As with unmixed Beauville structures, the concept of a mixed Beauville structure can be `internalised' to the group.

\begin{definition}\label{mixeddef}
Let $G$ be a group. A \emph{mixed Beauville structure} of $G$ is a set $\{x,y,z\}\subset G$ that generates an index 2 subgroup $G_0<G$ such that for every $g\in G\setminus G_0$ we have
\begin{enumerate}
\item $xyz=1$;
\item $\Sigma(x,y,z)\cap\Sigma(x^g,y^g,z^g)=\{1\}$ and
\item $g^2\notin\Sigma(x,y,z)$ for $i=1,2$.
\end{enumerate}
\end{definition}

Clearly no simple group can possess an mixed Beauville structure, however this doesn't rule out the possibility of an almost simple group possessing one. Few examples of mixed Beauille sructures are known \cite{Barker} so finding more is of great interest.

The following easy lemma is extremely useful.

\begin{lemma}\label{mixed}
Let $(\mathcal{C}\times\mathcal{C})/G$ be a Beauville surface of mixed type and $G_0$ the subgroup of $G$
consisting of the elements which preserve each of the factors, then the order of any element in
$G \setminus G_0$ is divisible by 4.
\end{lemma}

\begin{proof}
See \cite[lemma 5]{FG}.
\end{proof}

Of the 27 sporadic groups thirteen of them (namely M$_{12}$, M$_{22}$, J$_2$, $^2$F$_4(2)'$, HS, J$_3$, McL, He, Suz, O'N, Fi$_{22}$, HN and Fi$'_{24}$) posses outer automorphisms. From their character tables, which can be reconstructed from the data given in \cite{ATLAS}, we see that, apart from the Tits group $^2$F$_4(2)'$, all of the almost simple groups whose derived subgroup is in the above list have involutions lying outside $G_0$ and so by the above lemma none of these groups can possess a mixed Beauville structure. 

The case of the almost simple Tits group $^2$F$_4(2)$ is more delicate. In this case we see from the character table \cite[p.75]{ATLAS} that every element that is outside the simple group has order divisible by 4 and so lemma \ref{mixed} cannot be used to block the existence of a mixed Beauville structure. We can, however, also see the following. Condition 3 of definition \ref{mixeddef} forces the orders $x$, $y$ and $z$ to be odd since every involution of $G$ has the property that there is an element of order 4 in $G\setminus G_0$ that squares to it. The only elements of odd order have order 3, 5 or 13 and in each case there is only one class of cyclic subgroups of that order making it impossible to satisfy condition 2 of definition \ref{mixeddef}. The group $^2$F$_4(2)$ thus has no mixed Beauville structure, proving theorem \ref{unmixed}.

\section{The Alternating Groups}

In this final section we prove theorem \ref{alt} and corollary \ref{cor}. To do this we first recall some standard facts about automorphisms of alternating groups.

If $n\not=2,3$ or 6 then $Aut(A_n)\cong S_n$, the full symmetric group. (If $n=2,3$ then $Aut(A_n)\cong S_{n-1}$.) If $n=6$ then we have that $S_6$ is an index 2 subgroup of $Aut(A_6)$ which has structure $A_6:2^2$. An immediate consequence of this fact is the result that $Aut(A_6)$ has three index 2 subgroups, each of structure $A_6:2$. One is isomorphic to the linear group $PGL_2(9)$ (the exceptional isomorphism $A_6\cong PSL_2(9)$ gives us the fact that $Aut(A_6)\cong P\Gamma L_2(9)$); another to $S_6\cong P\Sigma L_2(9)$ and the final one to the Mathieu group $M_{10}$.

\begin{proof}\emph{of theorem \ref{alt}}.
Consider the permutations
\[x_1:=(2,9,5,6)(3,4,7,8)\mbox{,  }y_1:=(1,3,8,5)(2,6,10,4),\]
\[x_2:=(1,9,4,6,2)(3,5,7,10,8)\mbox{,  }y_2:=(1,3,2,5,7)(4,8,6,10,9),\]
and
\[g:=(1,10)(2,8)(3,6)(4,5)(7,9).\]

Easy calculation gives $o(x_1)=o(y_1)=o(x_1y_1)=4$ and $o(x_2)=o(y_2)=o(x_2y_2)=5$. Easy computations further show that $\langle x_1,y_1\rangle=\langle x_2,y_2\rangle=A_6$. From their orders it is clear that these elements also satisfy conditions 2 and 3 of definition \ref{maindef} and so these permutations define a Beauville structure for $A_6$ of type ((4,4,4),(5,5,5)). We claim that this Beauville structure is strongly real.

Easy computations show that $\langle x_1,y_1,g\rangle=\langle x_2,y_2,g\rangle=PGL_2(9)$, one of the groups of the form $A_6:2$ not isomorphic to the symmetric group $S_6$ (or the Mathieu group $M_{10}$). Further direct calculation reveals that $x_i^g=x_i^{-1}$ and $y_i^g=y_i^{-1}$ for $i=1,2$ and so this (outer) automorphism of $A_6$ shows that this Beauville structure is strongly real.
\end{proof}

\begin{proof}\emph{of corollary \ref{cor}}.
For $n\geq7$ these are explicitly constructed in the proof of \cite[theorem 2]{FG}. If $n=6$ this is the above theorem. If $n\leq5$ then it is easily verified that $A_n$ does not even possess a Beauville structure let alone a strongly real one.
\end{proof}

In \cite{FG} Fuertes and Gonz\'{a}lez-Diez use lemma \ref{mixed} to show that $S_6$ does not possess a mixed Beauville structure. In the case of $PGL_2(9)$ there are involutions lying outside the derived subgroup and so this same lemma ensures that $PGL_2(9)$ also does not possess a mixed Beauville structure. Remarkably, in the case of the group $M_{10}$ the only elements lying outside the derived subgroup all have order 4 or 8, so lemma \ref{mixed} is of no use here. 

In this case, however, we can say the following. Since $M_{10}$ has only one class of involutions it must be the case that, as in the case of $^2$F$_4(2)$, the elements defining a mixed Beauville structure must have odd order (ie order 3 or 5) by condition 3 of definition \ref{mixeddef}. Again, there is only one class of cyclic subgroups of each order, so Condition 2 of definiton \ref{mixeddef} cannot be satisfied, so there is no mixed Beauville structure in this case. The group $P\Gamma L_2(9)$ also cannot have an mixed Beauville structure since for each of the index 2 subgroups there is a class of involutions lying outside the subgroup blocking the existence of a mixed Beauville structure by lemma \ref{mixed}. It follows that no group of the form $A_6:2^{(2)}$ possesses a mixed Beaville structure.

\section{Acknowledgements} 
The author wishes to express his deepest gratitudes to Professor Gareth Jones for first introducing him to the concept of `all things Beauville'; to Dr John Bradley and Professor Christopher Parker for providing many helpful comments and corrections to early versions of this paper.

\end{document}